\documentclass[a4paper,12pt]{amsart}
\usepackage{color}
\usepackage{amsmath}
\usepackage{amssymb}
\usepackage{amsthm}
\usepackage[mathscr]{eucal}

\usepackage[usenames,dvipsnames]{xcolor}

\usepackage{graphicx}

\numberwithin{equation}{section}

\newcommand{\eps}{\epsilon}

\newcommand{\ga}{\gamma}
\newcommand{\si}{\sigma}

\newcommand{\Si}{\Sigma}
\newcommand{\ANR}{\operatorname{ANR}}
\newcommand{\CAT}{\operatorname{CAT}}

\newcommand{\R}{\operatorname{\mathbb{R}}}

\newcommand{\Sph}{\operatorname{\mathbb{S}}}

\theoremstyle{plain}
\newtheorem{thm}{Theorem}[section]
\newtheorem{lem}[thm]{Lemma}
\newtheorem{prop}[thm]{Proposition}
\newtheorem{cor}[thm]{Corollary}

\theoremstyle{definition}
\newtheorem{defn}[thm]{Definition}

\theoremstyle{remark}
\newtheorem{rem}[thm]{Remark}

\newtheorem{quest}[thm]{Question}

\title
[CAT(0) 4-manifolds]
{CAT(0) 4-manifolds are euclidean}
\author{Alexander Lytchak, Koichi Nagano, Stephan Stadler}



\begin{document}

\begin{abstract}
We prove that a topological 4-manifold of globally
non-positive curvature is homeomorphic to Euclidean space.
\end{abstract}

\maketitle

\section{Introduction}

\subsection{Main result}

This paper concerns the topology of {\em CAT(0) manifolds}.
These are synthetic generalizations of  complete, simply connected Riemannian manifolds of non-positive sectional curvature.
By the classical theorem of Cartan--Hadamard, any such Riemannian manifold is diffeomorphic to the Euclidean space $\R^n$.
In his seminal paper \cite{G_hypmfds}, Gromov asked if there exist simply connected  topological manifolds   other than Euclidean
space, which  admit a metric of non-positive curvature in a synthetic sense. (See Section~\ref{subsec: gen} for further discussion.)
The most important synthetic notion of non-positive curvature is the one due to Alexandrov.
The corresponding spaces were named {\em CAT(0)} by Gromov.
Any CAT(0) space is contractible. Thus, any CAT(0) 2-manifold is homemorphic to $\R^2$ by the classification of surfaces. In dimensions strictly greater than two,
contractible manifolds are Euclidean precisely when they are simply connected at infinity, thanks to classical topological results \cite{Freedman_4D, Husch, Stallings_pl_euc}.
In dimension three, CAT(0) manifolds are indeed Euclidean \cite{Br_mon, Ro_cell, T_tame}.
In dimensions strictly greater than four,
Davis and Januszkiewicz \cite{DJ} constructed examples of non-Euclidean CAT(0) manifolds. In this paper, we deal with the remaining open case:

\begin{thm}\label{thm_gromov}
Let $X$ be a CAT(0) space which is topologically a 4-dimensional
manifold. Then $X$ is homeomorphic to $\R^4$.
\end{thm}

This theorem also answers  the first question  in \cite[Section 6]{DJL_4D}.

\subsection{Related statements: the simplical case}
Examples of Davis and  Januszkiewicz \cite{DJ} mentioned above are simplicial complexes with piecewise Euclidean metrics. On the other hand,
 Stone \cite{Stone_pl}  verified that a \emph{PL} $n$-manifold which is CAT(0) with respect to a piecewise Euclidean metric is homeomorphic to the   Euclidean space $\R^n$. By the resolution of the Poincar\'e conjecture, any simplicial complex homeomorphic to a 4-manifold is a PL-manifold. Consequently,
any CAT(0) $4$-manifold with a piecewise Euclidean metric  is homeomorphic to $\R^4$.

 More recently, the question as to which  manifolds  carry a piecewise Euclidean CAT(0) metric has been reduced to a purely topological problem in  \cite{AB_collapse,AF_contractible}. Namely,  a manifold carries such a metric if and only if it is homeomorphic to a collapsible simplicial complex. Building on this result and previous work of Ancel--Guilbault \cite{AG_hyp}, it has been verified in \cite{AF_contractible} that  for all $n> 4$, the interior of any compact contractible $n$-manifold with boundary  carries a piecewise Euclidean CAT(0) metric.
Thus, in dimensions five and above, there is an abundance of contractible  CAT(0) manifolds.  Some necessary conditions for the existence of a piecewise Euclidean CAT(0) metric (hence for the existence of a collapsible triangulation, which is not PL)   are given in
\cite[Theorem 1]{AF_contractible}.
Motivated by Gromov's question, it is natural to ask the following, compare \cite[Question 3]{AF_contractible}.

\begin{quest}
Are there  CAT(0) topological manifolds which do not carry a piecewise Euclidean CAT(0) metric?
\end{quest}

\begin{quest}
What are  necessary and sufficient conditions  for the existence of a CAT(0) metric on a contractible manifold?
\end{quest}

\subsection{Related statements: the cocompact case}
Gromov's question has  been thoroughly studied in the \emph{cocompact} setting. (Recall, that
universal coverings of locally CAT(0) spaces are CAT(0) \cite{AB_ch}). 
By \cite[Theorem 5b.1]{DJ}, in
all dimensions strictly greater than four, there exist compact,  locally CAT(0) manifolds
 whose universal coverings are not Euclidean, see also \cite{ADG}. Moreover, in such dimensions there 
 exist compact,  locally CAT(0) manifolds which have no PL structure at all, \cite[Section 5a]{DJ}. We refer to \cite[Section 3]{DJL_4D} for an overview of further similar results.

In dimension four, several classes of  smoothable compact topological manifolds carrying a locally CAT(0) metric, yet not admitting a smooth metric of non-positive curvature, have been constructed in \cite{DJL_4D, Sat_4D, St_obst}.
In these examples, the universal covering is homeomorphic to $\R^4$ (as follows from our main theorem), and identifying the obstructions to the existence of smooth metrics relies on an intricate analysis of the group actions involved.

\subsection{Distance spheres}
Our proof depends upon an important contribution by Thurston \cite{T_tame}. He showed that
if all distance spheres to some fixed point $o\in X$ of a 4-dimensional CAT(0) manifold $X$
are topological 3-manifolds, then $X$ is homeomorphic to $\R^4$.  
Using a finer analysis of the metric structure of the space, we verify this latter condition in the more general setting of \emph{homology 4-manifolds}; see Sections \ref{subsec: hom} and \ref{subsec_hommfds} for the relevant definition and properties.

\begin{thm}\label{thm_main}
Let $X$ be a CAT(0) space which is a homology
$4$-manifold. Let $o\in X$ and $R>0$ be arbitrary. Then the distance sphere $S_R(o)$ of
radius $R$ around $o$ is a topological 3-manifold.
\end{thm}

We remark that this result does not hold true in dimensions $n \geq 5$, even
for piecewise Euclidean topological $n$-manifolds. Indeed, this can be seen in the aforementioned examples of Davis and Januszkiewicz; compare
\cite[Proposition~3d.3]{DJ}.

If $X$ is a CAT(0) $4$-manifold (and not just a homology $4$-manifold), then the resolution of the Poincar\'e conjecture together with  \cite{T_tame} implies that all distance spheres are homeomorphic to $\mathbb S^3$. Moreover, the homeomorphism in Theorem~\ref{thm_gromov}
is not completely abstract, but rather has the following geometric feature.

\begin{cor} \label{cor:th}
	Let $X$ be a $4$-dimensional CAT(0) manifold  and let $o\in X$ be an arbitrary point.  Then the distance function
	$d_o:X\setminus \{o \} \to (0,\infty)$ is a trivial fiber bundle with fiber $\mathbb S^3$.
\end{cor}

On the other hand, for a general CAT(0) homology $4$-manifold, the topology of the distance spheres may depend on the radius despite the fact that all of the spheres involved are manifolds. This can already be seen in the Euclidean cone $X=C(\Sigma)$ over the Poincar\'e homology sphere $\Sigma$.
The fine topological analysis of \cite{T_tame}, using the fact that the ambient space is a manifold, is therefore indispensable for the conclusion of our main theorem.

Corollary~\ref{cor:th} extends to the \emph{ideal boundary}
$\partial _{\infty } X$ and the \emph{natural compactification} $\bar X:=X\cup  \partial _{\infty } X$ of $X$, see \cite[Section II.8]{BH} for the definition and properties of ideal boundaries. In  dimensions $n\geq 5$, there are  CAT(0) spaces  homeomorphic to $\R^n$  with ideal boundary  different from $\mathbb S^{n-1}$ \cite[Theorem 5c.1]{DJ}.  In dimension four we show:

\begin{cor} \label{cor-jan}
	Let $X$ be a  $4$-dimensional topological manifold with a CAT(0) metric.  Then the 
	ideal boundary $\partial _{\infty} X$ of $X$ is homeomorphic to $\mathbb S^3$ and the 
	canonical compactification $\bar X=X\cup \partial _{\infty} X$ is homeomorphic to the closed unit ball in $\R^4$. 
\end{cor}

\subsection{Related statements: homology manifolds} \label{subsec: hom}
A homology $n$-manifold (without boundary)  is a locally compact metric  space $X$  of finite topological dimension such that,
for all  $x\in X$, the local homology $H_{\ast} (X,X\setminus \{x \})$ equals $H_{\ast} (\R^n,\R^n\setminus \{0 \})$.
 The structure theory of homology manifolds has been a central topic
in geometric topology for many decades and is  of fundamental importance in the topological \emph{manifold recognition} \cite{Can_rec,CHR_higher,Rep_rec}.

A homology $n$-manifold for $n\leq 2$ is always a topological $n$-manifold by a theorem of Moore \cite[Chapter IX]{rwilder}.
On the other hand,
in dimensions $n\geq 3$, a homology $n$-manifold may not have any manifold points at all \cite{DaWa_ghastly}. While there are deep and rather robust results allowing one to recognize when a homology manifold is a manifold in dimensions five and up, much less is known in dimensions three and four
\cite{Rep_rec}. Even though the tools of algebraic topology allow us to recognize homology manifolds in many instances, in particular, in the situation
of Theorem \ref{thm_main}  (in all dimensions), passing from homology manifolds to topological manifolds is difficult and requires some geometric insight.

  In the situation of Theorem~\ref{thm_main}, we achieve the needed local control of the topology of large spheres, by slicing them and  verifying that slices are $2$-dimensional spheres.
  Subsequently, these slices can be controlled uniformly with the help of Jordan's curve theorem. The control of the slices allows us to recover the local topology.

In contrast to the situation for general homology $n$-manifolds, CAT(0) homology $n$-manifolds are not too far from being manifolds. More precisely,
a CAT(0) homology $n$-manifold is a topological $n$-manifold on the complement of a discrete subset \cite[Theorem 1.2]{LN_top}; see \cite{Wu} for corresponding statements on spaces with lower curvature bounds.

We mention in passing a question of Busemann \cite{Bu_geo,   Ber-Bor}, which is related in spirit to the origins of this paper.
\begin{quest}
 Let $X$ be a locally compact geodesic  metric space. Assume that $X$ is geodesically complete and that there are no branching geodesics. Does $X$ have finite dimension? Is any such finite-dimensional $X$ a topological manifold?
\end{quest}

If such a space $X$ has finite dimension $n$, then  $X$ is a homology $n$-manifold and  if $n\leq 4$ then $X$ is a manifold
\cite{Bu_geo, Kra_3, T_buse, Ber-Bor}. For $n\geq 5$, the question remains open.
Finally, we mention that an answer to Busemann's question would follow from a purely topological conjecture of Bing and Borsuk \cite{BingBorsuk}.

\subsection{Minor generalizations} \label{subsec: gen}
An application of \cite[Theorem 1.1]{LS_improv}
extends
 Theorem~\ref{thm_gromov} and Corollary~\ref{cor:th} to other curvature bounds:
\begin{cor}
	Let $X$ be a $4$-dimensional topological manifold which is a $CAT(\kappa )$ space. Let $R>0$ be a real number with
	$R<\frac {\pi} {2\sqrt \kappa}$ if $\kappa>0$. Then, for any  $o\in X$, the open ball $B_R(o)$, the closed ball $\bar B_R(o)$ and the distance sphere $S_R(o)$
	are homeomorphic to the open unit ball in $\R^4$, to the closed unit ball in $\R^4$ and to $\mathbb S^3$, respectively.
\end{cor}

There are several notions of non-positive curvature for metric spaces. In all works cited above, Gromov's question has been studied for
CAT(0) spaces, even though the original question was posed for Busemann convex spaces, that is, geodesic spaces with a convex distance function.
Any CAT(0) space is Busemann convex. Conversely, examples of Busemann convex spaces that are neither CAT(0) nor  normed spaces are extremely rare,
compare \cite{IL}. Our ideas will apply
to this more general setting, once
 some structural results developed in \cite{LN_gcba} for CAT(0) spaces will be generalized to Busemann convex spaces.

\subsection{Comment on strategy and technique}
Our proof relies on the structural theory of geodesically complete spaces with upper curvature bounds developed in \cite{LN_gcba, LN_top}. Since any CAT(0) manifold is geodesically complete, this theory applies to the present situation. So-called `strainer maps', first appearing in \cite{BGP}
and defined by distance functions to points, are particularly useful. It has been verified \cite{LN_top} that for any point $o\in X$
as in Theorem \ref{thm_main}, all sufficiently small distance spheres around $o$ are (pairwise homeomorphic) 3-manifolds. In order to get sufficient control of remote spheres, an extension of the theory of strainer maps by one additional `orthogonal but non-straining' coordinate is required. This extension may be useful beyond  the present work. We refer the reader to Section~\ref{sec_relstr}  and Section~\ref{sec: last} for more details. Here, we only formulate
a special case of Proposition \ref{prop_contrhemi},  essential for the proof of Theorem~\ref{thm_main}.

\begin{thm} \label{thm_sph}
	Let $X$ be a locally compact, geodesically complete CAT(0) space, let $o\in X$, and
	let $p\in X$ be at distance $R>0$ from $o$. Then there exists $\epsilon >0$ such that for all $0<s<\epsilon$, the intersection of the distance sphere
	$S_s(p)$ with the closed ball $\bar B_R(o)$ is contractible.
\end{thm}

\includegraphics[scale=0.5,trim={0cm 1cm 0cm 2cm},clip]{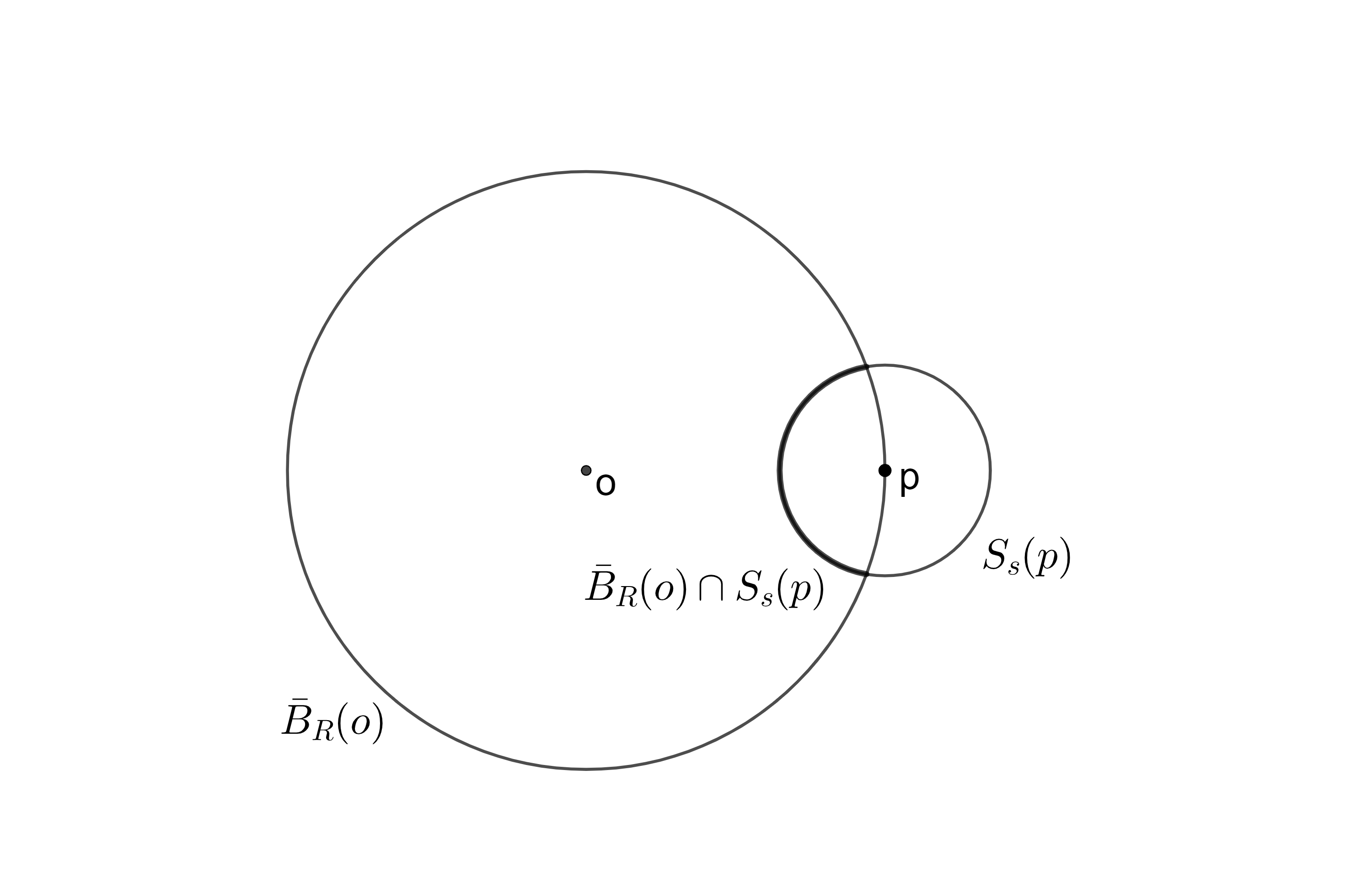}

\subsection{Acknowledgments}
Alexander Lytchak and Stephan Stadler were  supported by DFG grant SPP 2026. Koichi Nagano is supported by JSPS KAKENHI Grant Number 20K03603.
 The authors are grateful to Ronan Conlon for useful comments and to Tadeusz Januskiewicz
 for helpful suggestions including Corollary~\ref{cor-jan}.

\section{Preliminaries}

\subsection{Metric spaces}
We refer the reader to \cite{AKP,BBI,BH} for general background.
We denote by $d$ the distance in a metric space $X$. For  $x \in  X$,
 we denote by $d_x$ the distance function  $d_x(\cdot)=d(x,\cdot)$.
For $x\in X$ and $r>0$, we denote by $B_r(x)$ and $\bar B_r(x)$ the open, respectively, closed $r$-ball around $x$.
Similarly, $B_r(A)$ denotes the open $r$-neighborhood of a subset $A\subset X$.
Moreover, $S_r(x)$ denotes the $r$-sphere around $x$ and by $\dot B_r(x)$ we denote the punctured $r$-ball $B_r(x)\setminus\{x\}$.
For $\lambda>0$, we denote by  $ \lambda \cdot  X$ the metric space resulting from $X$ by rescaling the
metric by $\lambda$.
A \emph{geodesic}
is an isometric embedding of an interval. A \emph{triangle} is a union of three geodesics connecting three points.
$X$  is \emph{a  geodesic metric space} if
any pair of   points of $X$
is connected by a geodesic.
It is \emph{geodesically complete} if every isometric embedding of an interval extends to a
locally isometric embedding of $\R$.

A map $F:X\to Y$ between metric spaces is called \emph{$L$-Lipschitz} if
$d(F(x),F(\bar x))\leq L \cdot d(x,\bar x)$,  for all $x,\bar x \in X$.

The map $F$ is called
\emph{$L$-open} if the following condition holds.
For any   $x\in X$  and any $r>0$ such  that the closed ball  $\bar B_{Lr} (x)$ is complete,
 we have the inclusion
$B_{r} (F(x)) \subset F(B_{Lr} (x))$.

An \emph{ANR} will denote an \emph{absolute neighborhood retract}.
 For finite dimensional metric spaces, the only case relevant here, being an \emph{ANR} is equivalent to being locally contractible \cite{Hu}.

\subsection{Spaces with an upper curvature bound}

For $\kappa \in \R $, let $R_{\kappa}  \in (0,\infty] $ be the diameter of the  complete, simply connected  surface $M^2_{\kappa}$
of constant curvature $\kappa$. A complete  metric space is called a \emph{CAT($\kappa$) space}
if any pair of its points with distance $<R_{\kappa}$ is connected by a geodesic and if
 all triangles
with perimeter $<2R_{\kappa}$
are not thicker than
 the \emph{comparison triangle} in $M_{\kappa} ^2$.
A  metric space
is called a \emph{space with  curvature bounded above by  $\kappa$}
if   any point has a  CAT($\kappa$) neighborhood.
We refer to \cite{AKP,BBI,BH} for basic facts about such spaces.

For any CAT($\kappa$) space  $X$,
the angle between each pair of geodesics starting at the same point
is well defined. The \emph{space of directions}
 $\Sigma _x X$ at  $x\in X$,  equipped with the angle metric,
is a CAT(1) space.
The Euclidean cone
over $\Sigma _x$ is a CAT(0) space.
It is denoted by $T_x X$ and called
the \emph{tangent space} at $x$ of $X$. Its tip will be denoted by $o_x$.

 Let $x,y,z$ be three points at pairwise distance $<R_{\kappa}$ in a CAT($\kappa$) space $X$.
Whenever $x\neq y$, the geodesic between $x$ and $y$ is unique and will be denoted
by $xy$.   For $y,z \neq x$, the angle at $x$ between $xy$ and $xz$
will be denoted by $\angle yxz$.

In a CAT($\kappa$) space $X$, all balls of radii  smaller $  \frac {R_{\kappa}} 2$ are convex, hence $X$ is locally contractible. In fact, $X$ is an ANR   \cite[Theorem 3.2]{Linus}.

\section{Geometric topology}

\subsection{Homology manifolds} \label{subsec: homman}
 Denote by $D^n$ the closed unit ball in $\R^n$.

Let $M$ be a locally compact, separable metric space
of finite topological dimension.
We say that
$M$ is a
{\em homology $n$-manifold with boundary}
if
for any $p \in M$ we have a point $x \in D^n$
such that
the local homology
$H_{\ast}(M,M \setminus \{ p \})$ at $p$
is isomorphic to $H_{\ast}(D^n,D^n \setminus \{ x \})$.
The {\em boundary} $\partial M$  of $M$
is defined as the set of all points
at which the $n$-th local homologies are trivial.
In the case where the boundary of $M$ is empty,
we simply say that
$M$ is a
{\em homology $n$-manifold}.

If $M$ is a homology $n$-manifold with boundary then $\partial M$
is a closed subset of $M$ and it is  a homology $(n-1)$-manifold
by \cite{Mitch}.

Any homology $n$-manifold (with boundary) has dimension $n$. For $n\leq 2$,  we have   the following   theorem of Moore \cite[Chapter IX]{rwilder}.
\begin{thm} \label{thm: bing}
Any homology $n$-manifold with $n\leq 2$ is a topological manifold.
\end{thm}

A \emph{homology $n$-sphere} is a homology $n$-manifold $X$ with the same homology as the $n$-sphere:    $H_{\ast} (X)=H_{\ast} (\mathbb S^n)$.

\subsection{Uniform local contractibility}

A function $\rho:[0,r_0) \to [0,\infty )$ is called a \emph{contractibility function}
if it is continuous at $0$ with $\rho (0)=0$ and for all $t\in[0,r_0)$ holds $\rho (t) \geq t$ \cite{P_gh}.

 \begin{defn} \label{defn: glob}
We say that a family  $\mathcal F$ of metric spaces is \emph{uniformly locally contractible}
if  there exists a contractibility function $\rho :[0,r_0) \to [0,\infty )$ such that for any space $X$ in the family $ \mathcal F$,
any point $x\in X$ and any $0<r<r_0$, the ball $B_r(x)$ is contractible within the ball $B_{\rho (r)} (x)$.
\end{defn}

For example, the family of all  $\CAT(\kappa )$ spaces is uniformly locally contractible with $\rho :[0, \frac {\pi }{\sqrt  \kappa} ) \to [0,\infty)$ being the identity map.

Here is a special case of   \cite[Theorem~A]{Petersen}, \cite[Theorem~9]{P_gh}:
\begin{thm} \label{thm: petersen}
For any natural number $n$ and any family $\mathcal F$ of uniformly locally contractible metric spaces of dimension at
most $n$,
there exists some $\delta  >0$ such that any  pair of spaces $X,Y \in \mathcal F$,
with  Gromov--Hausdorff distance  at most $\delta $ is homotopy equivalent.
\end{thm}

The homotopy equivalences and the corresponding homotopies in Theorem~\ref{thm: petersen} can be chosen arbitrary close to the identity \cite{P_gh}.

When dealing with a family of   fibers of  some map, we will use the following more convenient variant of Definition~\ref{defn: glob} \cite{Ungar}.

\begin{defn} \label{defn: loc}
Let $F:X\to Y$ be a  map between metric spaces.
We say that  \emph{$F$ has uniformly locally contractible fibers}
if the following condition holds true for any point $x\in X$ and every neighborhood $U$ of $x$ in $X$.
There exists a neighborhood $V\subset U$ of $x$ in $X$ such that for any fiber
$F^{-1}(y)$ with non-empty intersection $F^{-1} (y)\cap V$, this intersection is contractible in $F^{-1} (y)\cap U$.
\end{defn}

 For    $X$ compact, a map $F:X\to Y$ has uniformly locally contractible fibers in the sense of Definition~\ref{defn: loc}
if and only if the family of the fibers is uniformly locally contractible in the sense of Definition~\ref{defn: glob}.

\subsection{Fibrations and fiber bundles}
A   map  $F:X\to Y$ between metric spaces is called   a \emph{Hurewicz fibration}
if it  satisfies the homotopy lifting property with respect to all spaces \cite[Section 4.2]{Hatcher}, \cite{Ungar}.

The map $F$ is called open if the images of open sets are open.

We will use the following result to recognize Hurewicz fibrations.

\begin{thm}[{\cite[Theorem 2]{Ferry_strongly}; \cite[Theorem 1]{Ungar}}] \label{thm: hurewicz}
Let $X,Y$ be finite-dimensional,   compact metric spaces and let $Y$ be an  $\ANR$.
Let $F:X\to Y$ be an open, surjective   map   with  uniformly locally contractible fibers.
Then  $X$ is an ANR and $F$ is a Hurewicz fibration.
\end{thm}

In some situations, Hurewicz fibrations turn out to be fiber bundles.
We will rely on the following.

\begin{thm} [{\cite[Theorems 1.1-1.4]{Ferry_alex}; \cite[Theorem 2]{Raymond-fibering}}] \label{thm: trivialbundle}
Let $X,Y$ be finite-dimensional locally compact $\ANR$s.
Let $F:X\to Y$ be a Hurewicz fibration. If all fibers of $F$ are closed $n$-manifolds
then $F$ is a locally trivial fiber bundle.
\end{thm}

\subsection{CAT(0) (homology) manifolds}\label{subsec_hommfds}

Following \cite{LN_gcba}, we will denote a locally compact, locally geodesically complete, separable space with an upper curvature bound
as GCBA. In this paper we are concerned with CAT(0) spaces which are homeomorphic to (homology) manifolds. We will call such spaces
{\em CAT(0) homology manifolds} and {\em CAT(0) manifolds} respectively. Every CAT(0) homology manifold is geodesically complete  \cite[Theorem~1.5]{LS_affine} and therefore GCBA.
Hence, we can rely on the results from \cite{LN_gcba, LN_top}.
For the local arguments of \cite{LN_gcba, LN_top}, the notion of a {\em tiny ball} played a role.
We point out that in a CAT(0) homology  manifold, a tiny ball is any ball of radius at most one. After rescaling, the  bound of 1 becomes irrelevant.

From \cite[Lemma~3.1, Corollary~3.4, Theorem~6.4]{LN_top} we infer:

\begin{prop} \label{prop_links}
Let $X$ be a CAT(0) homology $n$-manifold.
Then any space of directions $\Si_x X$ is a homology $(n-1)$-sphere. If $n\leq 4$, then
$\Si_x X$ is a topological manifold.
\end{prop}

Any CAT(0) homology $n$-manifold is a topological $n$-manifold on the complement of a discrete subset \cite[Theorem 1.2]{LN_top}.  For $n\leq 3$,  a CAT(0)  homology $n$-manifold is a manifold homeomorphic to $\R^n$
\cite[Theorem 6.4]{LN_top}, \cite{T_tame}.
The Euclidean cone  over the Poincar\'e sphere is
a CAT(0) homology $4$-manifold which is not  a manifold.

Any CAT(0) homology  $n$-manifold is locally bilipschitz equivalent to $\R^n$ away from a closed set of Hausdorff dimension at most $n-2$, as follows from \cite[Theorem 1.2 and Section 10.2]{LN_gcba}.

\section{Strainer maps} \label{sec: strainer}
We  recall the definition and basic properties of strainer maps in the framework of CAT(0) spaces from  \cite{LN_gcba} and \cite{LN_top}.
 Originally,
strainer maps were introduced in \cite{BGP} to study Alexandrov spaces with curvature bounded below.

\subsection{Almost spherical directions}
Let $X$ be a locally compact and geodesically complete CAT(0) space. Let $v$ be a direction at a point $x\in X$. An {\em antipode} of $v$ is a direction $\hat v\in\Sigma_x X$ at distance at least $\pi$ from $v$. If $v$ has a unique antipode $\hat v$, then $\Sigma_x X$ splits isometrically as a spherical join
$\Sigma_x X\cong\{v,\hat v\}\ast Z$. More generally,  the subset $\Sigma^0$ of points with unique antipodes in $\Sigma _x X$  is isometric to  $\mathbb S^k$, for  some $k$, and $\Sigma ^0$ is a spherical join factor of
$\Sigma_x X$
 \cite[Corollary~4.4]{Lbuild}.

A quantitative version is provided by the notion of $\delta$-spherical points and tuples.
The direction $v\in\Sigma_x X$ is called {\em $\delta$-spherical}, if there exists some $\bar v\in \Sigma_x X$ such that
for any $w\in \Sigma_x X$
\[d(v,w)+d(w,\bar v) <\pi +\delta.\]
Moreover, we say that  $v$ and $\bar v$ are  \emph{opposite} $\delta$-spherical points.
A $\delta$-spherical direction $v$ has a set of antipodes of diameter at most $2\delta$, \cite[Lemma~6.3]{LN_gcba}.
Therefore, if $\delta$ is small, this `almost leads to a splitting' of $\Sigma_x X$ \cite[Proposition~6.6]{LN_gcba}.

 A   $k$-tuple  $(v_1,\ldots,v_k)$ of
points in $\Sigma_x X$ is called   \emph{$\delta$-spherical}
if there exists another $k$-tuple $(\bar v_i)$ in $\Sigma_x X$ with the following two properties.
\begin{itemize}
	\item For $1\leq i \leq k$, the directions  $v_i$ and $\bar v_i$  are opposite $\delta$-spherical.
	\item For  $1\leq i\neq j \leq k$, the distances
   $ d(v_i,\bar v_j), d(v_i, v_j), d(\bar v_i,\bar v_j)$ are less than  $\frac \pi 2 + \delta$.
\end{itemize}

Moreover,  $(\bar v_i)$ and $(v_i)$ are called \emph{opposite $\delta$-spherical $k$-tuples}.

\subsection{Strainers and strainer maps}
A $k$-tuple $(p_1,\ldots,p_k)$ is called a {\em $(k,\delta)$-strainer} at a point $x\in X\
\setminus\{p_1,\ldots,p_k\}$ if the starting directions $v_i \in \Si _xX $ of the geodesics
$xp_i$ constitute a $\delta$-spherical  $k$-tuple in $\Sigma_x X$.

Two $(k,\delta)$-strainers  $(p_i)$ and $(q_i)$  at $x$
are  \emph{opposite} if the corresponding $\delta$-spherical  $k$-tuples $(v_i)$ and $(w_i)$ are opposite in $\Sigma _x X$.

A $k$-tuple $(p_i)$ in $X$ is a  {\em $(k,\delta)$-strainer}  in  $A\subset X\setminus\{p_1,\ldots,p_k\}$ if $(p_i)$ is a  $(k,\delta)$-strainer at every point $x\in A$.

The set of all points $U\subset X\setminus\{p_1,\ldots,p_k\}$
at which $(p_i)$ is a $(k,\delta)$-strainer is open in $X$  \cite[Corollary 7.9]{LN_gcba}.

Each $k$-tuple $(p_i)$ yields a {\em distance map} $F:X\to \R^k$ via $F=(d_{p_1},...,d_{p_k})$.
If $(p_i)$ is a $(k,\delta)$-strainer on a subset $A\subset X$, then the associated distance map  $F$
is called a {\em $(k,\delta)$-strainer map} on $A$.

\subsection{Properties of strainer maps}
For $\delta   \leq \frac 1 {4\cdot k}$ and $L= 2\sqrt k$,
every $(k,\delta)$-strainer map $F:U\to \R^k$ on an open subset $U\subset \R^k$
is $L$-open and $L$-Lipschitz  \cite[Lemma~8.2]{LN_gcba}.

The building blocks for straining maps are the following two observations.
First, for any $\delta>0$ and any  $x\in X$ the function $d_x:\dot B_r(x)\to(0,r)$ is a $(1,\delta)$-strainer map if
$r$ is chosen small enough \cite[Proposition~7.3]{LN_gcba}. Secondly, let $F:U\to\R^k$ be a $(k,\delta)$-strainer map and let $p$ be a point in a fiber $\Pi$ of $F$.
Then there exists  $r>0$ and a neighborhood
$W$ of  $\dot B_r(p)\cap \Pi$ in $U$ such that the map $\hat F=(F,d_p) :W\to \R^{k+1}$
is a $(k+1, 4\delta)$-strainer map \cite[Proposition~9.4]{LN_gcba}.

Any $(k,\delta)$-strainer map on an open subset of a $k$-dimensional CAT(0) space $X$ provides a bilipschitz chart \cite[Corollary 11.2]{LN_gcba}. In general,  we have the following topological structure.

\begin{thm}[{\cite[Theorem~5.1 and Corollary~5.2]{LN_top}}] \label{thm: fibration}
	Let $U$ be an open subset of a GCBA space   $X$. Let $F:U\to \R^k$ be a $(k,\delta)$-strainer map, for some  $k$ and
	some $\delta <\frac {1} {20 \cdot k}$.
	Then any  $x\in U$ has arbitrary small  open contractible neighborhoods $V$, such that the restriction $F:V\to F(V)$ is a Hurewicz fibration with contractible fibers.

	If a fiber $F^{-1}  (b)$ is compact, then there exists an open neighborhood $V$ of $F^{-1} (b)$ in $U$ such  that
	$F:V\to F(V)$ is a Hurewicz fibration.

	If $U$ is a homology $n$-manifold, then any fiber    $F^{-1} (b)$ is a homology $(n-k)$-manifold.
\end{thm}

\section{Extended strainer maps}\label{sec_relstr}
\subsection{Definition and basic properties}
Throughout this section, $X$ will denote a locally compact and geodesically complete CAT(0) space.

	Let $(p_1,..., p_k)$ be a $k$-tuple in $X$ and let $q\in X$ be an additional point.
	We say that $(p_1,...,p_k,q)$ is an \emph{extended $(k,\delta)$-strainer} in a subset  $A\subset X\setminus \{p_1,...,p_k,q\}$, if the following holds true for all $x\in A$:

The $k$-tuple $(p_i)$ is a $(k,\delta)$-strainer at $x$  and any continuation $qq'$ of the geodesic $qx$ beyond $x$ is such that,
 for all $1\leq i\leq k$,
\[\angle qxp_i < \frac{\pi}{2} + \delta \; \; \text{and}  \; \; \angle q'xp_i < \frac \pi 2 + \delta \,.\]

By the semi-continuity of angles, the set $U$ of all points  at which  $(p_1,...,p_k,q)$ is an extended $(k,\delta)$-strainer is open in  $X\setminus \{p_1,...,p_k,q\}$ \cite[Section~3.3 and Corollary~7.9]{LN_gcba}.

Let $(p_1,...,p_k,q)$ be an extended $(k,\delta)$-strainer in an open set $U\subset X$.  Then we call the map
\[\hat F= (d_{p_1},...,d_{p_k},d_q) = (F,d_q) : U\to (0,\infty)^{k+1}\]
an \emph{extended $(k,\delta)$-strainer map}.

By definition, an extended $(k,\delta)$-strainer map $\hat F:U\to \R^{k+1}$ is also an extended
	$(k,\delta')$-strainer map for any $0<\delta ' <\delta$.

\subsection{Basic properties}\label{subsec_basic}
Let $(p_1,...,p_k,q)$ be  an extended $(k,\delta)$-strainer at a point $x \in X$ and let  $qq'$ be an extension of the geodesic $qx$.
Since $\angle qxq' =\pi$,   for $1\leq i \leq k$, we have
\[\angle   p_i xq  > \frac \pi 2  -\delta \; \; \text{and}  \; \; \angle p_i x q'   > \frac \pi 2  -\delta \;.\]

 We fix  an opposite $(k,\delta)$-strainer $(p_1',...,p_k')$ to $(p_i)$ at the point $x$.  The definition of opposite strainers implies:
\[\angle   p_i ' xq  < \frac \pi 2  + 2\delta \; \; \text{and}  \; \; \angle p_i 'x q'   < \frac \pi 2  + 2\delta \;.\]
Therefore, we also get

\[\angle   p_i ' xq  > \frac \pi 2  - 2\delta \; \; \text{and}  \; \; \angle p_i 'x q'   > \frac \pi 2  - 2\delta \;.\]

Applying \cite[Lemma~8.1]{LN_gcba} (compare \cite[Lemma~8.2]{LN_gcba})  we get:
\begin{lem}\label{lem_almostiso}
For $\delta \leq \frac 1 {20\cdot k}$ and $L=  2\sqrt {k+1}$,  	any extended $(k,\delta)$-strainer map $\hat F: U\to \R^{k+1}$ is  $L$-Lipschitz and $L$-open.
\end{lem}

\begin{rem}
 The
  argument  in  \cite[Lemma~8.3]{LN_gcba}  allows to choose
  the  constant $L$ above  arbitrary close to $1$, if
only $\delta $ is sufficiently small.
\end{rem}

By definition, any point $q$ is an extended $(0,\delta)$-strainer in $X\setminus \{q \}$ for any $\delta >0$ and the distance function $d_q: X\setminus \{q\}  \to (0,\infty)$ is an extended $(0,\delta)$-strainer map.   We are interested in distance spheres, thus fibers of such $(0,\delta)$-strainer maps.
As in \cite{BGP,LN_gcba},  we approach the structure of these fibers by finding more strainers:

\begin{lem}\label{lem_extrelstr}
	Let $\hat F= (d_{p_1},...,d_{p_k}, d_q):U\to\R^{k+1}$ be an extended $(k,\delta)$-strainer map for some $k\geq 0$ and $\delta < \frac 1 {20\cdot k}$.
   Let $p\in U$ be arbitrary and let $\hat \Pi_p:=\hat F^{-1} (\hat F(p))$ be the fiber of $\hat F$ through $p$.

	Then there exists  $r>0$ such that $(p,p_1,...,p_k,q)$ is an extended
	$(k+1, 4\cdot \delta)$-strainer in the intersection of $\hat \Pi_p$ and the punctured ball
	$\dot B_r(p)$.
\end{lem}

\begin{proof}
 We apply \cite[Proposition~9.4]{LN_gcba} and find some $r>0$ such that
 $(p,p_1,...,p_k)$ is a $(k+1, 4\delta)$-strainer in   $\hat \Pi_p \cap \dot B_r(p)$.

 For any $x\in \hat \Pi_p \cap \dot B_r(p)$, the points $x$ and $p$ are
 at equal distance to $q$,
  hence  $\angle p xq <\frac \pi 2$.
It remains to prove that, for  sufficiently small $r>0$,   $\angle p xq'<\frac \pi 2+ \delta$, for all $x\in \hat \Pi_p \cap \dot B_r(p)$
and all points $q'$ with $\angle qxq'=\pi$. Suppose for contradiction that we find
$x_i \in \hat \Pi_p \setminus \{p\}$ converging to $p$ and points $q_i'$ lying on extensions of the geodesics $qx_i$ such that
$\angle p x_iq_i'\geq\frac \pi 2+ \delta$. We may assume that $d(x_i,q_i')=1$ and that $q_i'$ converges to a point $q_\infty'$.
Using geodesic  completeness,  we extend $x_i p$ to a geodesic $x_i p_i$ of length one and such that
 \[\angle p_i p q_\infty'= \pi - \angle q_\infty' p x_i \leq \frac \pi 2\;.\]

\includegraphics[scale=0.6,trim={-3cm 2cm 0cm 4cm},clip]{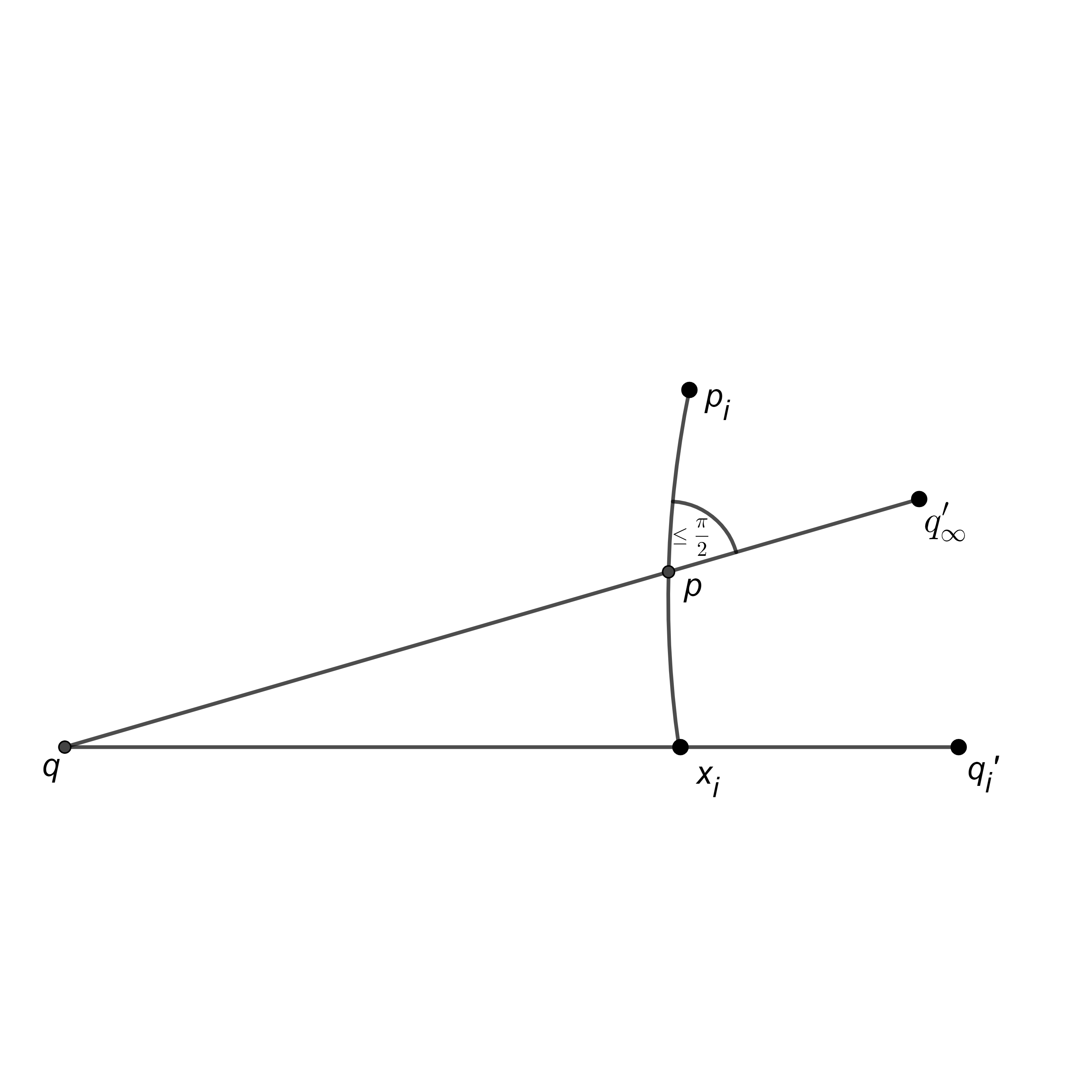}

Taking a subsequence, we may assume that $p_i$ converges to a point $p_{\infty}$. Hence, $\angle p_\infty p q_\infty'\leq\frac \pi 2$.
But by semi-continuity of angles
\[\angle p_\infty p q_\infty' \geq \limsup\limits_{i\to\infty} \angle p_i x_iq_i'\geq\frac \pi 2+ \delta\;.\]

This contradiction finishes the proof.
\end{proof}

\subsection{Halfspaces}
Let $\hat F=(F,d_o)$ be an extended strainer map on an open set $U$. We denote the $F$-fiber and $\hat F$-fiber through a
point $x$ by $\Pi_x$ and $\hat\Pi _x$, respectively. We define the {\em $\hat F$-halfspace} through $x$ by
\[\hat\Pi_x^{+}=\{y\in \Pi_x|\ d_o(y)\leq d_o(x)\}.\]

 The proof of our main results will rely on the following structural results about the  fibers and halfspaces of extended strainer maps.   The proofs of these results
 are postponed to Section~\ref{sec: last}.

 The first result generalizes   \cite[Proposition~2.7]{T_tame} and \cite[Corollary~5.2]{LN_top}:

\begin{prop}\label{prop_fiberhom}
		Let $U$ be an open subset of $X$. Then for
		any extended $(k,\delta)$-strainer map $\hat F: U\to \R^{k+1}$ with $\delta<\frac{1}{64\cdot k}$  the following holds true for any $x\in U$.
	\begin{enumerate}
		\item The halfspace $\hat \Pi _x ^{+}$ is an ANR.
		\item If $U$ is a homology $n$-manifold, then $\hat \Pi _x ^{+}$ is a homology $(n-k)$-manifold with boundary $\hat \Pi _x$. The fiber
		$\hat \Pi_x$ is a homology $(n-k-1)$-manifold without boundary.
	\end{enumerate}
\end{prop}

The second statement is an extension of Theorem~\ref{thm_sph} which constitutes the special case where $k=0$ and $y=x$. In the proof of our main results we will
only need the cases $k=0$ and $k=1$.

\begin{prop}\label{prop_contrhemi}
For every relatively compact set $V\subset X$ there exists $\delta_0 >0$ with the following property.
Let $\hat F: X\to \R^{k+1}$ be a distance map which is an extended $(k,\delta_0)$-strainer map at a point $x\in V$.
Then there exist $\epsilon_0, s_0 >0$ such that for any $0<\epsilon < \epsilon _0$  and any point $y$ with $d(x,y) <s_0\cdot \epsilon$
	 the `hemisphere' $S_{\epsilon} (x) \cap \hat \Pi^+_y$ is contractible and locally contractible.
\end{prop}

\section{Proof of the Main Theorem}
\subsection{Topology of intersecting spheres}
We begin the proof of  Theorem~\ref{thm_main}.
Thus, we fix a CAT(0) homology $4$-manifold $X$,  a point $o \in X$ and some radius $R>0$.
We denote by $S$ the distance sphere $S=S_R(o)$ and we are going to verify that $S$ is a topological $3$-manifold.

We fix an arbitrary $p \in S$ for the rest of the proof.  We need to find
 a neighborhood of $p$ in $S$ which is homeomorphic to $\R^3$.
For this
we aim  to show that  the restriction of $d_p$ to $S$ is a fiber bundle on a punctured
neighborhood of $p$ in $S$. The proof boils down to understanding how
distance spheres intersect in  our  CAT(0) homology 4-manifold $X$.

We apply Proposition~\ref{prop_contrhemi} to the relatively
 compact set   $V:= B_{2R}(o)$.
  Note that
$\hat F:=d_o :V\to \R$ is an extended $(0,\delta)$-strainer map for any $\delta >0$. The halfspace
$\hat \Pi _p ^+$ is exactly the ball $\bar B_R(o)$ and $\hat \Pi _p = S$.

\begin{cor}\label{cor_S2}
	There exists a radius $r_p>0$ such that  $S_r(p)\cap S$ is homeomorphic to $\Sph^2$
	for every $0<r\leq r_p$.
\end{cor}

\begin{proof}
By \cite[Proposition~9.4]{LN_gcba} and Lemma~\ref{lem_extrelstr}, we can	choose $r_p$  such that $p$ is a $(1,\delta)$-strainer in $\dot B_{r_0} (p)$ and
	$(p,o)$ is an extended $(1,\delta)$-strainer on $\dot B_{r_0} (p) \cap S$.
In addition, we choose $r_p$ smaller than the constant $\eps_0=\eps_0(p)$ from Proposition~\ref{prop_contrhemi}.
 Then
by  Proposition~\ref{prop_fiberhom} and Proposition~\ref{prop_contrhemi}, $S_r(p)\cap \bar B_R(o)$ is a contractible  homology  $3$-manifold with boundary
$S_r(p)\cap S$, for all $r<r_p$.
Thus,  $S_r(p)\cap S$ is a homology $2$-manifold and therefore a $2$-manifold by Theorem~\ref{thm: bing}.
By Poincar\'e duality,  $S_r(p)\cap S$ is a homology $2$-sphere, see \cite[Proposition 2.8]{T_tame}. Due to the classification of surfaces, $S_r(p)\cap S$ is homeomorphic to $\mathbb S^2$.
\end{proof}

\begin{lem} \label{lem: tech}
 There exists  $r_0=r_0(p)>0$  such that the distance function $d_p:\dot B_{r_0} (p)\cap S \to (0,r_0)$
	has uniformly locally contractible fibers.
\end{lem}

\begin{proof}
Let $\delta_0$ be the constant from Proposition~\ref{prop_contrhemi} and set   $\delta= \delta_0/4$. Let $r_p$ be as in Corollary~\ref{cor_S2}.
By \cite[Proposition~9.4]{LN_gcba} and Lemma~\ref{lem_extrelstr}, we can	choose $r_0< r_p$  such that $p$ is a $(1,\delta)$-strainer in $\dot B_{r_0} (p)$ and
	$(p,o)$ is an extended $(1,\delta)$-strainer on $\dot B_{r_0} (p) \cap S$.

We fix an arbitrary
$x \in  \dot B _{r_0} (p)\cap S$  and set $t_0:= d_p (x)$. In addition, we fix a positive number $\rho_0<r_0-t_0$.

\medskip

\noindent{\bf Sublemma.}
There exists a positive  $\epsilon_0 <\rho_0$ and a positive $s_0<1$
such that for  all $t$ with  $|t-t_0|<s_0\cdot\eps_0$  the intersection of spheres  $S_{\epsilon_0} (x)\cap S_t (p) \cap S$ is
homeomorphic to $\mathbb S^1$.
\medskip

 We apply  \cite[Proposition~9.4]{LN_gcba} and Lemma~\ref{lem_extrelstr} and find   $\epsilon_0 <\rho_0$ small enough, so that
$(p,x)$ is a $(2,\delta_0)$-strainer  in $\dot B_{2\epsilon_0}(x)\cap S_{t_0}(p)$ and $(p,x,o)$ is an extended $(2,\delta _0)$-strainer  in
$\dot B_{2\epsilon_0 }(x)\cap S_{t_0}(p)\cap S$.

Using the openness of the strainer property,
we find some small $s_0>0$, such that
for all $t$ with $|t-t_0|<s_0\cdot\eps_0$, the pair $(p,x)$ is a  $(2,\delta _0)$-strainer  in $\dot B_{2\eps_0}(x)\cap S_t(p)$  and the triple  $(p,x,o)$ is an extended $(2,\delta _0)$-strainer  in $\dot B_{2\eps_0}(x)\cap S_t(p)\cap S$.

We apply Proposition~\ref{prop_fiberhom} and deduce that  for all such $t$
the intersection
$S_{\eps_0} (x)\cap S_t(p)\cap \bar B_R(o)$ is a homology 2-manifold with boundary $S_{\eps_0} (x)\cap S_t(p)\cap S$.
By Theorem~\ref{thm: bing},
these intersections  $S_{\eps_0} (x)\cap S_t(p)\cap \bar B_R(o)$ are  2-manifolds with boundary $S_{\eps_0} (x)\cap S_t(p)\cap S$.

By our choice of $\delta _0= 4\cdot\delta$, we may apply Proposition~\ref{prop_contrhemi}.  By possibly making $\epsilon_0$ and $s_0$ even smaller, we deduce that all
intersections
$S_{\eps_0} (x)\cap S_t(p)\cap \bar B_R(o)$ are contractible and therefore homeomorphic to closed discs. Hence their boundaries
$S_{\eps_0} (x)\cap S_t(p)\cap S$ are circles. This finishes the proof of the  sublemma.
\medskip

Now we can easily finish the proof of the lemma.
By the choice of $r_0, \rho_0$ and Corollary~\ref{cor_S2}, any fiber $S_t(p)\cap S$  is homeomorphic to
$\mathbb S^2$.

In order to verify the uniform local contractibility of the fibers of the restriction of $d_p$,
we will argue that for every $t$ with $|t-t_0|<s_0\cdot\eps_0$, the set
$B_{\eps_0} (x)\cap S_t(p)\cap S$ is contractible inside
	$B_{\rho_0} (x)\cap S_t(p)\cap S$.

In the same parameter range as above, $B_{\eps_0} (x) \cap S_t(p) \cap S$ is an open subset of the 2-sphere $S_t (p) \cap S$ whose topological  boundary inside
$S_t(p) \cap S$ is contained in the circle $S_{\eps_0} (x)\cap S_t(p)\cap S$. Therefore, by the Jordan curve theorem,
$\bar B_{\eps_0}  (x) \cap S_t (p) \cap S$ is either  a topological disc or all of $S_t(p) \cap S$.
In both cases, $B_{\eps_0} (x) \cap S_t(p) \cap S$ is contractible inside $\bar B_{\epsilon_0} (x) \cap S_t (p) \cap S$.  Therefore,  $B_{\eps_0} (x) \cap S_t(p) \cap S$ is contractible inside the larger set $  B_{\rho_0} (x)\cap S_t(p)\cap S$.
\end{proof}

\subsection{The main results}
We can now finish:

\begin{proof}[Proof of Theorem~\ref{thm_main}]
	Let $p\in S=S_R(o)$ be arbitrary.  Choose $r_p$ as in Corollary \ref{cor_S2} and $r_0<r_p$ as in Lemma~\ref{lem: tech}.
The  distance function 	 $d_p:\dot B_{r_0}(p)\cap S\to(0,r_0)$  has uniformly locally contractible fibers homeomorphic to $\mathbb S^2$ by Corollary~\ref{cor_S2} and   Lemma~\ref{lem: tech}. By Lemma~\ref{lem_almostiso} and Theorem~\ref{thm: hurewicz}, $d_p$
	is a fiber bundle.
 Hence $\dot B_{r_0}(p)\cap S$ is homeomorphic to $\Sph^2\times(0,r_0)$.  Therefore,
 $B_{r_0}(p)\cap S$ is homeomorphic to a 3-ball. Since $p$ was arbitrary, $S=S_R(o)$ is a 3-manifold as required.
\end{proof}

Now the main result of \cite{T_tame} implies Theorem~\ref{thm_gromov}.
\medskip

Before turning to  Corollary~\ref{cor:th} and Corollary~\ref{cor-jan}, we recall some notation.
We fix a CAT(0) $4$-manifold $X$ and a point $o\in X$.

For $R>r>0$ we have a canonical \emph{geodesic contraction map} 
\[c_{R,r}:S_R(o)\to S_r(o)\] 
defined by sending $y\in S_R(o)$ to the 
point of intersection of the geodesic $oy$ with $S_r(o)$.    

The ideal boundary $\partial _{\infty} X$ is canonically identified with the \emph{inverse limit for the bonding maps} $c_{R,r}$   \cite[Section II.8.5]{BH}, \cite[p. 872]{Fuji}:
\[\partial _{\infty} X= \varprojlim S_r (o) \;.\]
The canonical map $c_{\infty ,r}:\partial _{\infty } X \to S_r(o)$, the inverse limit of the bonding maps $c_{R,r}$, sends
a point $\xi \in \partial _{\infty } X$ to the intersection with $S_r(o)$ of the ray starting in $o$ and determined by $\xi$ \cite[Section II.8]{BH}.

We  recall in a special case the notions of 
cell-like  mappings, referring  to  \cite{Mitchell, T_tame} for details. A compact subset $K$ of a $3$-manifold $M$ is called \emph{cell-like}, if $K$ is contractible in any neighborhood of $K$ in $M$. 
 A map
$f:M\to N$ between $3$-manifolds is called a \emph{cell-like map}
if  the preimage of any point is cell-like.

The following result is a combination of \cite[Theorem 1.2, Corollary 1.4]{Mitchell} and \cite[Corollary 2.2]{Mc2}:

\begin{thm} \label{lem: siebenmann}
Let $M$ be a compact $3$-manifold and $f:M\to M$ be a  surjective cell-like map.
Then $f$  
	 is a uniform limit of homeomorphisms.
	For any open subset $U\subset M$, the restriction
	$f:f^{-1} (U)\to U$ is a homotopy equivalence.	 
\end{thm}

Now we turn to the proof of   Corollary~\ref{cor:th}, essentially contained in the proof of
 \cite[Theorem 4.3]{T_tame}. 

\begin{proof}[Proof of Corollary~\ref{cor:th}]
By Theorem~\ref{thm_gromov}, Theorem~\ref{thm_main} and 
\cite[Theorem 4.3]{T_tame} all distance spheres $S_R(o)$ are homotopy equivalent to $\R^4\setminus \{o\}$, hence to $\mathbb S^3$.
By Theorem~\ref{thm_main} and the resolution of the Poincar\'e conjecture, any sphere $S_R(o)$ is homeomorphic to $\mathbb S^3$.

Hence, for any geodesic contraction 
 $c_{R,r}:S_R(o)\to S_{r} (o)$ the preimage of any point is contained in a subset homeomorphic to $\R^3$.
Combining  \cite[Corollary~2.10  and Theorem~2.13]{T_tame} and the subsequent remark, we deduce that $c_{R,r}$
   are cell-like maps.  
   
   By Theorem~\ref{lem: siebenmann}, for any open contractible set $W\subset  S_r(o)$, the preimages 
   $c_{R,r}^{-1} (W)$ are contractible for all $R>r$.   This implies that
   the map $d_o :X\setminus \{o\} \to (0,\infty)$ has  uniformly locally contractible fibers.

Hence an application of Theorem~\ref{thm: trivialbundle} completes the proof.
\end{proof}

In the proof of the Corollary \ref{cor-jan} below we assume some knowledge of the ideal boundary, its cone topology and the canonical compactification $\bar X= X\cup \partial _{\infty} X$ of 
a CAT(0) space $X$.

	\begin{proof}[Proof of Corollary \ref{cor-jan}]
		We fix a point $o$ in the CAT(0) $4$-manifold $X$.
		By Corollary~\ref{cor:th}, every distance sphere $S_R(o)$ is homeomorphic to 
		$\mathbb S^3$.  
		
		As we have seen in the proof of Corollary~\ref{cor:th} above,
		all 
		geodesic contractions $c_{R,r}:S_{R} (o) \to S_{r} (o)$ are cell-like maps.  
	Hence, due to Theorem~\ref{lem: siebenmann}, the map $c_{R,r}$ is a uniform limit of homeomorphisms.  
		Thus, an application of the main result of \cite{Brown} implies that the ideal boundary
		$\partial _{\infty } X$ is homeomorphic to $\mathbb S^3$.
		Moreover, the proof in 
		\cite{Brown} shows that the canonical projection 
		$$c_{\infty, r} :\partial _{\infty} X\to S_r(o)$$
		is a uniform limit of homeomorphisms.

		Consider the map $f:\bar X  \to (0,1]$ which sends $\partial _{\infty} X$   to $1$ and is defined as  $$f(x):= \frac   {d_o(x)} {1+d_o (x)} $$
		on $X$. The map $f$ is continuous on $\bar X$ and coincides on $X$ with $d_o$ up to a homeomorphism of the image interval. Hence $f$ is a fiber bundle on $X\setminus \{o \}$.  We claim
		that $f$ is also a fiber bundle on all of $\bar X\setminus \{o\}$.  
		
		All fibers of $f$ are $3$-spheres. Moreover, for any open contractible  $W$ in any sphere $S_r(o)$ the preimages $c^{-1} _{R,r}	 (W)$ are contractible, for all $R\in [r, \infty]$, since all geodesic contractions including $c_{\infty,r}$ are cell-like.
		This implies that  $f$ has uniformly locally contractible fibers.
		
		 Theorem~\ref{thm: trivialbundle} implies the claim. Therefore, $\bar X\setminus \{o\}$
		 is homeomorphic to $\mathbb S^3\times (0,1]$. Thus, $\bar X$, the one-point-compactification of this space, is homeomorphic to the $4$-ball. 
	\end{proof}

\section{Structure of fibers of extended strainer maps} \label{sec: last}

\subsection{Generalized distance functions}

In this final section, we want to use information on limits of distance maps to conclude topological properties
of their fibers. This requires a slight generalization of the notion of distance functions and strainer maps.
For this purpose we make the following definitions, see also \cite[Section~5]{N_asym}. Recall that a convex function on a CAT(0) space attains its minimal value on a closed convex set or doesn't attain
a minimum at all.
A {\em generalized distance function} on a CAT(0) space $X$ is a convex function $b:X\to\R$ whose (negative) gradient has unit norm on the complement of its
minimal set:
\[\|\nabla_x (-b)\|:=\max\left\{0,\limsup_{y\to x}\frac{b(x)-b(y)}{d(x,y)}\right\}\equiv 1.\]
This definition unifies the concept of distance functions to convex subsets and Busemann functions.  Adding a constant to a generalized distance function results in a generalized distance function.
On every bounded open set, a generalized distance function equals the distance function to a convex set up to a constant.
In particular, the integral curves of the `negative gradient' of a generalized distance function are geodesics and the {\em negative gradient}
$\nabla_x (-b)\in\Si_x X$ is well-defined.

A map $F:X\to \R^k$  will be called a {\em generalized distance map}, if all coordinates $f_i$ of $F$ are generalized distance functions.

A \emph{generalized distance map} $F:X\to \R^k$  with components $f_i$ will be called a generalized $(k,\delta)$-strainer map in a subset $A\subset  X$ if the minimum sets of any  $f_i$ is disjoint from $A$  and the following holds true. For any $x\in A$, the negative gradients
$ \nabla_x (-f_i)\in\Si_x X$ form a $\delta$-spherical $k$-tuple  of directions in $\Si _x X$.

Two generalized distance maps $F,\bar F:X\to \R^k$ with coordinates $f_i, \bar f_i$ are
\emph{opposite
generalized $(k,\delta)$-strainer maps} on $A\subset X$, if for all $x\in A$, the corresponding
$k$-tuples  of negative gradients are opposite $(k,\delta)$-strainers.

  As for (non-generalized) distance maps, the set of points $x\in X$,  at which a generalized distance map $F:X\to \R^k$ is a generalized $(k,\delta)$-strainer map is open. Similarly, the set
  of points at which  a pair of distance maps $F, \bar F$ are opposite generalized $(k,\delta)$-strainers
  is open.

Let $F:X\to\R^k$ be a generalized distance map with coordinates $f_i$ and denote by $b$ another generalized distance function.
Suppose that $F$ is a generalized $(k,\delta)$-strainer map on a subset $A\subset X$ and $b$ does not attain its minimum on $A$.
Then the map $\hat F=(F,b):X\to\R^{k+1}$ is called a {\em  generalized extended $(k,\delta)$-strainer map} on $A$,
if at all points $x\in A$ and for every antipode $w_x\in\Si_x X$ of $\nabla_x(-b)$ the following holds, for all $1\leq i\leq k$:
\[\angle_x(\nabla_x(-b),\nabla_x(-f_i))<\frac{\pi}{2}+\delta \; \; \text{and}  \; \; \angle_x(w_x,\nabla_x(-f_i))<\frac{\pi}{2}+\delta.\]
The set of points where a given generalized distance map is a generalized extended $(k,\delta)$-strainer map is open, again due to the semi-continuity of angles.

All statements about (extended) strainer maps transfer to the generalized setting.
For instance, the concept of `straining radius' introduced in \cite[Section~7.5]{LN_gcba} generalizes as follows.
Let $F:X\to\R^k$ be a generalized distance map with coordinates $f_i$.
Suppose that $F$ is a generalized  $(k,\delta)$-strainer map at a point $x$.
Then the {\em straining radius} is the  largest radius
$\si_x$ with the following property. For every $y\in B_{\si_x}(x)$ and $1\leq i\leq k$ let $\bar p_i$ be any point with $d(x,\bar p_i)=1$
and such that the direction at $x$ of the geodesic $x\bar p_i$ is antipodal to $\nabla(-f_i)$. Then, $F$ and $\bar F=(d_{\bar p_1},\ldots,d_{\bar p_k})$
are opposite generalized $(k,2\delta)$-strainer maps on $B_{\si_x}(y)$. The proof of positivity of $\si_x$ is identical to \cite[Lemma~7.10]{LN_gcba}.
Similarly, we define an `extended straining radius'.
If $\hat F=(F,b)$ is an extended generalized $(k,\delta)$-strainer map at $x$, then the {\em extended straining radius}
is the largest radius  $\hat\si_x\leq\si_x$ such that for all
$y\in B_{\hat\si_x}(x)$ the map $\hat F$ is an extended generalized $(k,2\delta)$-strainer map on $B_{\hat\si_x}(y)$.

Let $(X_n,x_n)$ be a sequence of  pointed locally compact CAT(0) spaces converging in the
pointed Gromov--Hausdorff topology to a space $(X,x)$, Then, for  any sequence of generalized distance functions $f_n:X_n \to \R$ with uniformly bounded $f_n (x_n)$, we find a subsequence converging to a generalized distance  function $f:X\to \R$.

Let $F_n:X_n\to \R^k$ be a sequence of generalized distance maps converging to a generalized distance map $F:X\to \R$.   The semi-continuity of angles under convergence implies
the following, as in \cite[Lemma~7.8]{LN_gcba}. If $F$ is a generalized (extended) $(k,\delta)$-strainer at $x$ then $F_n$ is a generalized (extended) $(k,\delta)$-strainer at $x_n$, for all $n$ large enough.  Moreover,
for all $n$ large enough, the (extended) straining radius $\sigma _{x_n}$ of $F_n$ at $x_n$ is
bounded from below by half of the (extended) straining radius $\sigma _x$.

\subsection{Local topology of halfspaces}
The following result on strainer maps translates to the generalized setting as well. But since we only apply it in the non-generalized setting, and
since this allows us to directly rely on  \cite[Theorem~9.1]{LN_gcba}, we refrained from formulating a generalized version
even though proofs extend literally.

\begin{prop}\label{prop_defohs}

	Let $\hat F=(F,d_o):X\to\R^k$ be a distance map and $\delta\leq\frac{1}{64\cdot k}$. Suppose that $\hat F$ is an extended $(k,\delta)$-strainer map at a point $x$
	with extended strainer radius $\hat\si_x$.
	Denote by  $W$ a ball $B_{r}(x)$ with radius $r\leq\hat\si_x$. Then  there
	exists  a deformation retraction
	of $W$ onto the halfspace $\hat\Pi_x^{+}\cap W$.
\end{prop}

\includegraphics[scale=0.4,trim={2cm -1cm 0cm 14cm},clip]{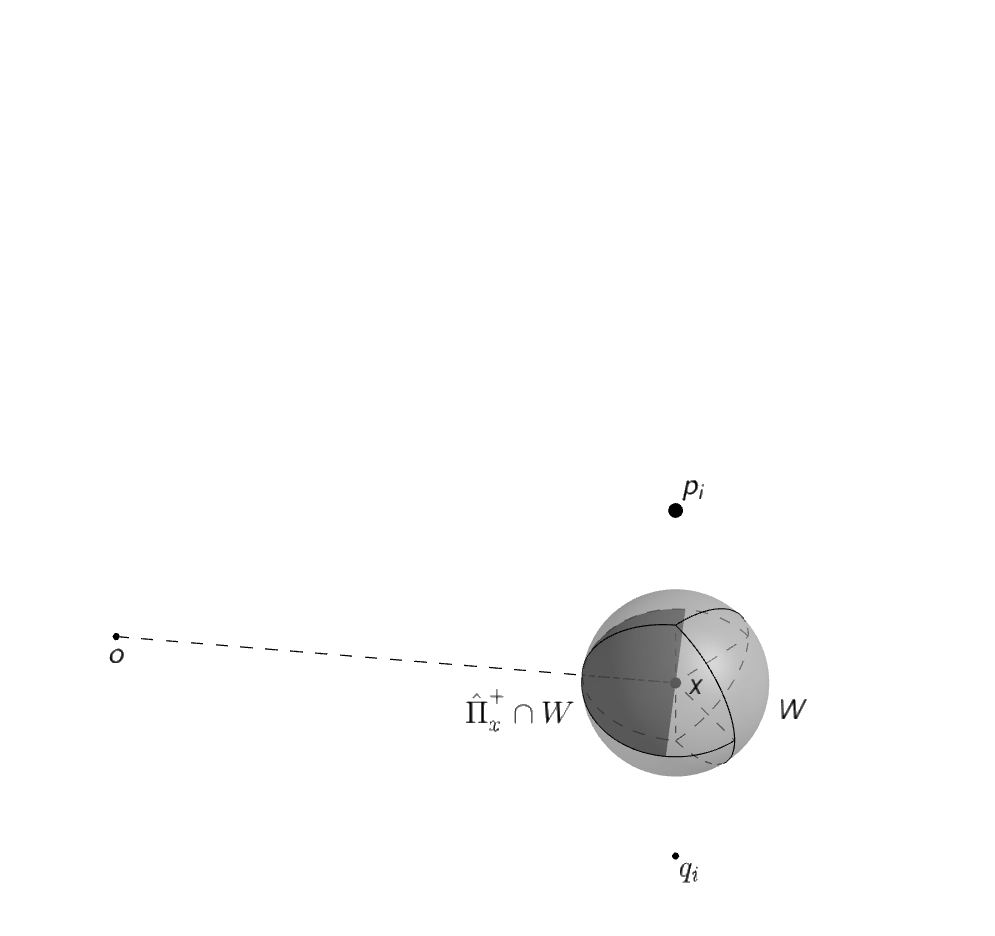}

\begin{proof}
The proof is an adaption of the proof of \cite[Theorem~9.1]{LN_gcba}.
For convenience of the reader, we stick to the notation of \cite[Theorem~9.1]{LN_gcba}.
Hence, $(p_i)$ denotes the strainer defining $F$ and $(q_i)$ is a $k$-tuple in $X\setminus W$
such that $x$ lies on the geodesic $p_i q_i$ and the tuples $(p_i)$ and $(q_i)$ are opposite $(k,2\delta)$-strainers in $W$.

Set $s_0:=d(o,x)$ and define the function $M:W\to\R$ by
\[M(z):=\max\limits_{1\leq i\leq k}|d(p_i,z)-d(p_i,x)|.\]

Denote by  $\Phi:W\times[0,1]\to W$ the homotopy which retracts $W$ onto $\Pi_x\cap W$, provided by \cite[Theorem~9.1]{LN_gcba}. Recall that  the length of the path $\ga_z(t):=\Phi(z,t)$ is at most $8k\cdot M(z)$.
Moreover, $\ga_z(t)$ is an infinite piecewise geodesic all of whose segments are directed towards one of the points $p_i$ or $q_i$.  By the first variation formula,  the value of $d_o$ changes along $\ga_z(t)$ with velocity at most $4\delta$.
Hence, for all $t\in [0,1]$,
\begin{equation}
|d_o(z)- d_o(\gamma _z(t))| \leq 4\delta \cdot 8k \cdot M(z) \leq \frac 1 2\cdot M(z)\;.
\label{eq:d_o_est}
\end{equation}

Denote by $\varphi:W\times[0,1]\to W$ the flow which deformation retracts $W$ onto $\bar B_{s_0}(o)\cap W$. More precisely, $\varphi$
moves a point $z\in W\setminus \bar B_{s_0}(o)$ towards $o$ at unit speed until it reaches $S_{s_0}(o)$ and then stops.
Note that  $\varphi$ does indeed preserve $W$, by the CAT(0) property of $X$.
We define a concatenated homotopy $\Psi$ by setting
$\Psi(z,t)=\Phi(z,2t)$ for $t\leq\frac{1}{2}$ and  $\Psi(z,t)=\varphi(\Phi(z,1),2t-1)$ for $t\geq\frac{1}{2}$.

By definition, $d(o,\Psi(z,1))\leq s_0$, for all $z\in W$ and $\Psi$ fixes $\hat\Pi_x^{+}\cap W$.

The length of the $\Psi$-flow line of a general point $z\in W$ is bounded above by $8k\cdot M(z)+2r$.
However, if $d(o,z)\leq s_0$ holds, then the length of the $\Psi$-flow line  starting at $z$ is at most $(1+4\delta)\cdot 8k\cdot M(z)$ by (\ref{eq:d_o_est}).

 Along the homotopy
 $\varphi$ the value of $M$ changes at most with velocity $4\delta$, due to the first variation formula.  Hence, for any $z$ with $d(o,z) \leq s_0$,
we deduce, using $M(\Phi (z,1))=0$ and
(\ref{eq:d_o_est}):
 \[M(\Psi (z,1)) \leq 2\delta\cdot  M(z) \;.\]

To obtain the required deformation retraction, we take a limit of iterated concatenations of $\Psi$. More precisely, for $m\geq 1$, we define homotopies  $\Psi_m:W\times[0,1]\to W$ as follows. The homotopy  $\Psi_m$ is the identity on the interval $[1-2^{-m}, 1]$  and it equals
a rescaling of $\Psi$ on any of the intervals $[1-2^{-l}, 1-2^{-l-1}]$, for $l=0,...,m-1$.

The above inequalities  imply  $M(\Psi _m (z,1))  \leq (2\delta)^m \cdot M(z)$, by induction.
Moreover, the flow line of $\Psi _m$  starting at $z\in W$ has length uniformly bounded above by $2r+24k\cdot M(z)$.
Therefore, $(\Psi_m)$ converges uniformly to a homotopy $\Psi_\infty:W\times[0,1]\to W$ as required.
\end{proof}

Recall that a closed subset $\Pi$ of a topological space $Y$ is called  {\em homotopy negligible
in $Y$} if for each open set $U$ of $Y$ the inclusion $U\setminus\Pi\to U$ is a homotopy
equivalence. If $Y$ is an ANR, this condition is satisfied, if any point $z\in \Pi$ has a
neighborhood basis  $\mathcal U_z$  of contractible neighborhoods $U_z$ with contractible  complements  $U_z \setminus \Pi$, \cite[Theorem 1]{eells}.
 In our setting, we have:

\begin{cor}\label{cor_Z}
Let $\hat F=(F,d_o):U\to\R^k$ be an
extended $(k,\delta)$-strainer map on an open set $U\subset X$ with $\delta\leq\frac{1}{64\cdot k}$.
Then, for every $x\in U$, the halfspace $\hat\Pi_x^+$ is an ANR and the fiber $\hat\Pi_x$ is homotopy negligible in  $\hat\Pi_x^+$.
\end{cor}

\begin{proof}
By \cite[Theorem~9.1]{LN_gcba}, for all $y\in\hat\Pi^+_x\setminus\hat \Pi_x$
the set  $\hat\Pi_x^+\cap B_{r}(y)$ is contractible  as a retract of $B_{r}(y)$, as long as
	the radius $r$ is less than the straining radius $\si_y$ and the difference of levels $d_o(x)-d_o(y)$.
Similarly, 	by Proposition~\ref{prop_defohs}, for all $z\in \hat\Pi_x$, the set  $\hat\Pi_x^+\cap B_{r}(z)$ is contractible  as a retract of $B_{r}(z)$
	for some radius $r<\hat\si_{z}$. Hence, $\hat \Pi _x ^+$ is an ANR.
Now for $z\in \hat\Pi_x$ set $W=B_{\hat\si_z}(z)$ as above.
It remains to show that $(\hat\Pi^+_x\setminus\hat \Pi_x)\cap W$ is contractible. Since it
is an ANR as an open subset of $\hat\Pi^+_x$, it suffices to verify that all of its homotopy groups vanish, \cite[Corollary VII.8.5]{Hu}.
This will follow, once we have shown that for any compact subset $K\subset (\hat\Pi^+_x\setminus\hat \Pi_x)\cap W$, the inclusion map $K\hookrightarrow (\hat\Pi^+_x\setminus\hat \Pi_x)\cap W$ is
nullhomotopic. By continuity of the straining radius $\hat\si_x$, for a given such set $K$, we find a point $w\in(\hat\Pi^+_x\setminus\hat \Pi_x)\cap W$
and $s\leq\hat\si_w$ with $K\subset B_s(w)\subset W$. By Proposition~\ref{prop_defohs}, $\hat\Pi^+_w\cap B_s(w)\subset (\hat\Pi^+_x\setminus\hat \Pi_x)\cap W$ is contractible and the proof is complete.
\end{proof}

\begin{proof}[Proof of Proposition~\ref{prop_fiberhom}]
	We have already seen in Corollary~\ref{cor_Z} that the halfspace  $\hat \Pi^+_x$ is an ANR.

		Assume now that $U$ is a homology $n$-manifold.
	 Since $\delta<\frac{1}{20\cdot k}$, Theorem~\ref{thm: fibration} implies that  the fibers of $F$ are homology  $(n-k)$-manifolds.
	The complement of $\hat\Pi_x$
	in $\hat\Pi_x^+$ is open in $\Pi_x$ and therefore a homology $(n-k)$-manifold. By Corollary~\ref{cor_Z}, $\hat\Pi_x$ is homotopy negligible  in  $\hat\Pi_x^+$. In particular,
	every singleton $\{y\} \subset \hat\Pi_x$ is homotopy negligible in  $\hat\Pi_x^+$ \cite[Corollary~2.6]{tor}.
	We conclude that the local homology groups $H_\ast(\hat\Pi_x^+,\hat\Pi_x^+\setminus\{y\})$
	vanish at all points $y\in\hat\Pi_x$. By \cite{Mitch}, $\hat\Pi_x$ is the boundary of a homology manifold and therefore is itself a homology manifold
	without boundary.
\end{proof}

\subsection{Contractibility of hemispheres}

At last, we provide

\begin{proof}[Proof of Proposition~\ref{prop_contrhemi}]

Suppose for contradiction that there is a sequence $\delta_l\to 0$ and distance maps $\hat F_l=(F_l,d_{q^l}): X\to \R^{k+1}$
with $F_l=(d_{p^l_1},\ldots,d_{p^l_k})$
which are extended $(k,\delta_l)$-strainer maps at points $x_l\in V$ where the statement fails.
Thus we find arbitrary small `hemispheres' around $x_l$ which are either not contractible
or not locally contractible.
More precisely, we find sequences $\eps_l\to 0$, $s_l\to 0$ and a sequence of points $y_l\in X$ with $d(x_l,y_l)< s_l\cdot\eps_l$ and
the following additional properties.
\begin{enumerate}
	\item({\em Gromov--Hausdorff close to tangent space})
	\[|\bar B_{\eps_l\cdot l}(x_l),\bar B_{\eps_l\cdot l}(o_{x_l})|_{GH}<\frac{\eps_l}{l};\]
	\item({\em improved strainer}) the map $(d_{x_l},F_l)$ is a $(k+1,4\cdot\delta_l)$-strainer map on $\dot B_{\eps_l\cdot l}(x_l)$;
	\item({\em improved extended strainer}) the map $(d_{x_l},\hat F_l)$ is an extended $(k+1,4\cdot\delta_l)$-strainer map on an open neighborhood $V_l$
	of $\hat\Pi_{x_l}\cap \dot B_{\eps_l\cdot l}(x_l)$;
	\item({\em large levels})
	\[\min\{d_{p^l_1}(x_l),\ldots,d_{p^l_k}(x_l),d_{q^l}(x_l)\}\geq l\cdot\eps_l;\]
	\item ({\em fiber lies in extended domain})
	\[S_{\epsilon_l} (x_l) \cap \hat \Pi_{y_l}\subset V_l;\]
	\item({\em non-contractible})  $S_{\epsilon_l} (x_l) \cap \hat \Pi^+_{y_l}$ is either not contractible
or not locally contractible.
\end{enumerate}

The first item can be arranged because in our setting tangent spaces are Gromov--Hausdorff limits of rescaled balls around a particular point
\cite[Corollary~5.7]{LN_gcba}.
The second and third item follow from  \cite[Proposition~9.4]{LN_gcba} and Lemma~\ref{lem_extrelstr}, respectively, by choosing $\eps_l$ small enough.
Similarly, the forth item can be achieved by choosing $\eps_l$ small enough. Finally,
the fifth item can then be guaranteed by choosing $s_l$ small enough.

	We define the shifted strainer maps $\Phi_l=(d_{p_1^l}-d_{p_1^l}(x_l),\ldots,d_{p_k^l}-d_{p_k^l}(x_l))$,
	as well as the shifted distance function $b_l=d_{q^l}-d_{q^l}(x_l)$.
	In particular, $\Phi_l(x_l)
	=0$.
	Now we rescale space and functions by $\frac{1}{\eps_l}$. Since $V$ is relatively compact, up to passing to subsequences, we can take a pointed Gromov--Hausdorff limit
	$(X_\infty,x_\infty)=\lim\limits_{l\to\infty} (\frac{1}{\eps_l}\cdot X,x_l)$ \cite[Proposition~5.10]{LN_gcba}.
	We also pass to corresponding limits of functions:
	$\Phi_\infty =\lim\limits_{l\to\infty}\frac{1}{\eps_l}\cdot \Phi_l$
	and
	$b_\infty =\lim\limits_{l\to\infty}\frac{1}{\eps_l}\cdot b_l$.
Item four above ensures that all coordiantes of $\Phi_\infty$, as well as the function $b_\infty$, are  Busemann functions on $X_\infty$ \cite[Lemma~2.3]{KL_haken}.

	By condition (1) above, $(X_\infty,x_\infty)$ is isometric to a pointed Gromov--Hausdorff limit of the sequence of tangent spaces
	$(T_{x_l}X,o_{x_l})$. In particular, $(X_\infty,x_\infty)$ is isometric to a Euclidean cone with tip $x_\infty$. Therefore,
	$S_1(x_\infty)$ is a CAT(1) space \cite{Ber_metrization}. Moreover, the spaces of directions $\Si_{x_l} X$ converge to $S_1(x_\infty)$ \cite[Theorem~13.1]{LN_gcba}.
By assumption, the negative gradients of the components of $\Phi_l$ provide a $\delta_l$-spherical $k$-tuple of	directions at $x_l$.
\cite[Proposition~6.6]{LN_gcba} implies that
$S_1(x_\infty)$ splits isometrically as a spherical join $S_1(x_\infty)\cong\Sph^{k-1}\ast\Si'$;
equivalently, $X_\infty$ splits isometrically as a direct product $X_\infty\cong \R^k\times X'_\infty$.
Moreover, the negative gradients of the components of $\Phi_\infty$ form a spherical $k$-tuple inside the $\Sph^{k-1}$-factor of
$S_1(x_\infty)$.

From the $L$-openness of $(d_{x_l},\hat F_l)$ (Lemma~\ref{lem_almostiso}),
	we conclude the Gromov--Hausdorff convergence
	$\lim\limits_{l\to\infty}\frac{1}{\eps_l}\cdot(S_{\eps_l}(x_l)\cap \hat\Pi^+_{y_l})=S_{1}(x_\infty)\cap\hat\Pi^+_{x_\infty}$.

	\medskip

\noindent{\bf Sublemma.}
The sequence $\frac{1}{\eps_l}\cdot(S_{\eps_l}(x_l)\cap\hat\Pi^+_{y_l})$ is
	uniformly locally contractible.
	\medskip

	By compactness of $S_{1}(x_\infty)$, we find $r>0$ such that the straining radius of $\Phi_\infty$ satisfies $\si_z>2r$ at all
	$z\in  S_{1}(x_\infty)\cap \Pi_{x_\infty}$
	and the extended straining radius of $(\Phi_\infty,b_\infty)$ satisfies $\hat\si_z>2r$ at all $z\in S_{1}(x_\infty)\cap \hat\Pi_{x_\infty}$.
	Then, for $l$ large enough, the straining radius of $F_l$ and the extended straining radius of $(F_l,b_l)$ is larger than $r\cdot\eps_l$ on
	$S_{\epsilon_l} (x_l) \cap \Pi_{y_l}$ and $S_{\epsilon_l} (x_l) \cap \hat\Pi_{y_l}$, respectively.

	Let $z_l\in S_{\epsilon_l} (x_l) \cap \hat\Pi^+_{y_l}$ be a point.
	If the distance from $z_l$ to the fiber $S_{\epsilon_l} (x_l) \cap \hat\Pi_{y_l}$ is at least $\frac{r}{2}\cdot\eps_l$, then
	$S_{\epsilon_l} (x_l) \cap \hat\Pi^+_{y_l}\cap B_{\frac{r}{2}\cdot\eps_l}(z_l)$ is contractible by \cite[Theorem~9.1]{LN_gcba}.
	On the other hand,
	if the distance from $z_l$ to the fiber $S_{\epsilon_l} (x_l) \cap \hat\Pi_{y_l}$ is smaller than $\frac{r}{2}\cdot\eps_l$, then
	$S_{\epsilon_l} (x_l) \cap \hat\Pi^+_{y_l}\cap B_{\frac{r}{2}\cdot\eps_l}(z_l)$ is contained in an ball $B_{r\cdot\eps_l}(w_l)$
	with $w_l\in S_{\epsilon_l} (x_l) \cap \hat\Pi_{y_l}$. By Proposition~\ref{prop_defohs}, the set
	$S_{\epsilon_l} (x_l) \cap \hat\Pi^+_{y_l}\cap B_{r\cdot\eps_l}(w_l)$ is contractible. It follows that any $\frac{r}{2}$-ball
	in $\frac{1}{\eps_l}\cdot(S_{\epsilon_l} (x_l) \cap \hat\Pi^+_{y_l})$ is contractible inside its concentric $2r$-ball.

	\medskip

		As one consequence, the same local contractibility holds for the hemispheres $S_{1}(x_\infty)\cap \hat\Pi^+_{x_\infty}$
		\cite[Theorem~9]{P_gh}.    Moreover, $S_{1}(x_\infty)\cap \hat\Pi^+_{x_\infty}$
			is homotopy equivalent to  $S_{\epsilon_l} (x_l) \cap \hat \Pi^+_{y_l}$,  for large enough $l$.  Hence, to arrive
		 at a contradiction, it remains to show that
		  $S_{1}(x_\infty)\cap \hat\Pi^+_{x_\infty}$ is contractible.

	Let $v\in S_1(x_\infty)$ denote the point corresponding to the negative gradient of  $b_\infty$. By semi-continuity of angles
	and the splitting $S_1(x_\infty)\cong\Sph^{k-1}\ast\Si'$, we see $v\in\Si'$ (cf.~Section~\ref{subsec_basic}). In particular, $b_\infty=b_\infty'\circ\pi'$ where $\pi'$ denotes the projection $X_\infty\cong \R^k\times X'_\infty\to X'_\infty$,
	and $b_\infty'$ is a Busemann function on $X'_\infty$. We infer
	$\hat\Pi^+_{x_\infty}\cong\{0\}\times\{b_\infty'\leq 0\}$ since $\Phi_\infty(x_\infty)=0$. Hence,
	\[S_{1}(x_\infty)\cap \hat\Pi^+_{x_\infty}\cong\Si_{x_\infty}X_\infty'\cap\{b_\infty'\leq 0\}.\]

	But $\Si_{x_\infty}X_\infty'\cap\{b_\infty'\leq 0\}=\bar B_{\frac{\pi}{2}}(v)\subset \Si_{x_\infty}X_\infty'$ and therefore
$\Si_{x_\infty}X_\infty'\cap\{b_\infty'\leq 0\}$ is contractible, since
	$\Si_{x_\infty}X_\infty'$ is CAT(1). Consequently, the hemisphere $S_{1}(x_\infty)\cap \hat\Pi^+_{x_\infty}$ is contractible. Contradiction.
\end{proof}

\bibliographystyle{alpha}
\bibliography{R4}


\end{document}